\documentclass[12pt]{article}
\usepackage{graphicx,epsfig}
\usepackage{color}
\usepackage{amsmath,amssymb,mathrsfs,amsthm}
\usepackage{amsfonts}
\usepackage{multicol}
\usepackage{graphicx}
\usepackage{cite}
\usepackage{color}
\usepackage{pdfcolmk}
\usepackage[nice]{nicefrac}
\usepackage{amsthm}
\usepackage{latexsym}
\usepackage{dsfont}
\usepackage{setspace}
\usepackage{subfigure}
\usepackage[a4paper, tmargin=1.5in, bmargin=1in, lmargin=1in, rmargin=1in, headheight=13.6pt]{geometry}
\usepackage[sort&compress,numbers]{natbib}
\usepackage[fit]{truncate}
\usepackage[font=small,labelfont=bf,labelsep=space]{caption}
\makeatletter

\def\@citex[#1]#2{\if@filesw\immediate\write\@auxout{\string\citation{#2}}\fi
  \def\@citea{}\@cite{\@for\@citeb:=#2\do
    {\@citea\def\@citea{,\linebreak[0]\hskip0pt plus .2em}%
      \@ifundefined{b@\@citeb}%
    {{\bf ?}\@warning{Citation `\@citeb' on page \thepage\space undefined}}%
      \hbox{\csname b@\@citeb\endcsname}}}{#1}}
\newtheorem{theorem}{Theorem}[section]

\newtheorem{problem}{Problem}[section]
\newtheorem{rule-def}[theorem]{Rule}
\numberwithin{equation}{section}
\makeatletter

\begin{document}
\title{Haar wavelet quasilinearization technique for doubly singular boundary value problems}
\author{Randhir Singh\thanks{Corresponding author. E-mail:~{randhir.math@gmail.com} },
Himanshu Garg\thanks{E-mail:~{himanshugargcse@gmail.com} } and
Apoorv Garg \thanks{E-mail:~{apoorv.garg.cse15@gmail.com} } \\
\small $^{*}$ Department of Mathematics, Birla Institute of  Technology  Mesra,  India. \\
\small $^\dag$ Department of Computer Science Engineering, Birla Institute of  Technology  Mesra,  India.\\
}

 \maketitle{}
\begin{abstract}
 \noindent The Haar wavelet based quasilinearization technique for solving a general class of singular  boundary value problems
  is proposed. Quasilinearization technique is used to linearize nonlinear singular problem. Second rate of convergence is obtained of a sequence of linear singular problems. Numerical solution of linear singular problems is obtained by Haar-wavelet method.
  In each iteration of quasilinearization technique, the numerical solution is updated by the Haar wavelet method. Convergence analysis of Haar wavelet method is discussed. The results are compared with the results obtained by the other technique and
with exact solution. Eight singular problems are solved to show the applicability of the Haar wavelet  quasilinearization technique.
\end{abstract}
\section{Introduction}
In this paper, we consider the following  class of nonlinear doubly singular  boundary value problems (DSBVPs) \cite{Bobisud1990,singh2014approximate,singh2014adomian}
\begin{eqnarray}\label{sec:eq1}
(p(x)y'(x))'= q(x)f(x,y(x)),~~~~~~~~0<x<1,
\end{eqnarray}
with Dirichlet  boundary conditions (BCs.)
\begin{align}
\label{sec:eq2} &y(0)=\alpha_1,~~~y(1)=\beta_1,
\end{align}
and Neumann-Robin  BCs.
\begin{align}
\label{sec:eq3} y'(0)=0,~~~\alpha_2\; y(1)+\beta_2\; y'(1)=\gamma_2,
\end{align}
where $\alpha_1$, $\beta_1$, $\alpha_2$, $\beta_2$ and $\gamma_2$ are any real constants.
Here, $p(0)=0$ and $q(x)$ may be discontinuous at $x=0$. Throughout this paper, the following conditions are assumed on $p(x)$, $q(x)$ and $f(x,y)$:\vspace{-0.5cm}
\begin{description}
\item $(C_1)$ $p(x)\in C[0,1]\cap C^{1}(0,1]$ and $p(x)>0,~q(x)>0\in(0,1]$.
\item $(C_2)$ $ \displaystyle \frac{1}{p(x)}\in L^{1}(0,1]$ and $ \displaystyle \int\limits_{0}^{1}\frac{1}{p(x)}\int\limits_{x}^{1}q(s)ds\; dx<\infty$.  (for boundary conditions \eqref{sec:eq2})
\item $(C_3)$   $q(x)\in L^{1}(0,1]$ and
$\displaystyle  \int\limits_{0}^{1}\frac{1}{p(x)}\int\limits_{0}^{x}q(s) ds\; dx<\infty$.~~ (for boundary conditions \eqref{sec:eq3})
 \item $(C_4)$ $f(x,y),~f_{y}(x,y)\in C(\Omega)$ and $f_{y}(x,y)\geq 0$ on  $ \Omega$, where $ \Omega:=\{(0,1]\times \mathbb{R}\}$.
\end{description}
The well known Thomas-Fermi equations \cite{thomas1927calculation,fermi1927metodo}, is modeled by the problem \eqref{sec:eq1}, where $p(x)=1$,  $q(x)=x^{-\frac{1}{2}}$
 and $f=y^{\frac{3}{2}}.$ The problem \eqref{sec:eq1}, where $p(x)=q(x)=x^{2}$  arises in oxygen diffusion in a spherical cell \cite{lin1976oxygen,anderson1980complementary} with  $f$ of the form $$f(x, y)= \frac{n y}{y+k},~~n>0,~~k>0,$$
and in modelling of heat conduction in human head \cite{flesch1975distribution,gray1980distribution,duggan1986pointwise}
 with $f$ of the form $$f(x, y)=-\delta e^{-\theta y},~~~\theta>0,~~~\delta>0.$$
 Existence and uniqueness of doubly singular  boundary value problems \eqref{sec:eq1} with BCs.  \eqref{sec:eq2} and \eqref{sec:eq3} can be found in  \cite{chawla1987existence,dunninger1986existence,Bobisud1990,pandey2009note}. In general, such singular problems are difficult to solve due its singular behavior at $x=0$. There are several techniques to solve doubly singular  boundary value problems \eqref{sec:eq1} with BCs. \eqref{sec:eq3} where $p(x)=q(x)=x^{\alpha}$ for $\alpha>0$ or $0<\alpha<1$. The numerical study  of doubly singular  boundary value problems  has been carried out for past couple of decades and still it is an active area of research to develop some better numerical schemes. So far various numerical
methods such as the collocation methods \cite{Reddien1973projection,russell1975numerical}, tangent chord method   \cite{duggan1986pointwise}, finite difference methods \cite{jamet1970convergence,chawla1982finite,chawla1984finite}, spline finite difference methods \cite{iyengar1986spline}, B-Spline method  \cite{kadalbajoo2007b}, spline method \cite{kumar2007higher},  Chebyshev economization method \cite{kanth2003numerical}, Cubic spline method   \cite{kanth2005cubic,kanth2006cubic,kanth2007cubic},  Adomian decomposition method (ADM) and modified ADM \citep{inc2005different,mittal2008solution,khuri2010novel,ebaid2011new,Kumar2010}, ADM with Green's function  \cite{singh2013numerical,singh2014efficient}, variational iteration method (VIM) \cite{wazwaz2011comparison,wazwaz2011variational,ravi2010he}, the optimal modified VIM  \cite{singh2017optimal}, homotopy analysis method  \cite{danish2012note,roul2017new} and  homotopy perturbation method \cite{roul2016new} and the references cited therein.

\noindent In the recent years the Haar wavelet technique  has been popular in the field of numerical approximations. The basic idea of  the Haar wavelets and its applications can be found in \cite{hsiao2001haar,hsiao2004haar,lepik2005numerical,lepik2007numerical,ur2012numerical,singh2016haar,babaaghaie2017numerical,aziz2010numerical,aziz2014new}. The Haar wavelets have gained popularity among researchers for their useful properties such as simple applicability, orthogonality and compact support. Compact support of the Haar wavelet
basis permits straight inclusion of the different types of boundary conditions in the numeric algorithms. Due to the linear
and piecewise nature, the Haar wavelet basis lacks differentiability and hence the integration approach will be used instead
of the differentiation for calculation of the coefficients.  Boundary value problems are considerably more difficult to deal with than initial value problems (IVPs).

\noindent  The Haar wavelet method for BVPs is more complicated than for IVPs.  The quasilinearization approach was introduced by
 Bellman and Kalaba \cite{bellman1965quasilinearization} to solve the individual or systems of nonlinear ordinary
 and partial differential equations.  The application of Haar wavelet method quasilinearization technique  for solving different models can be found in \cite{jiwari2012haar,kaur2013haar,saeed2013haar}.

\noindent In this work, an efficient numerical method based on Haar wavelets quasilinearization technique is  proposed for solving doubly singular  boundary value problems \eqref{sec:eq1} with BCs.  \eqref{sec:eq2}  and  \eqref{sec:eq3}.  The main aim of the present paper is to obtain numerical solutions of nonlinear doubly singular boundary value problems  over a uniform grids with a simple method based on the Haar wavelets and quasilinearization technique.  The quasilinearization technique is used to linearize nonlinear singular problem. Second rate of convergence is obtained of a sequence of linear singular problems. Numerical solution of linear singular problems is obtained by Haar-wavelet method. In each iteration of quasilinearization technique, the numerical solution is updated by the Haar wavelet method. Convergence analysis of Haar wavelet method is discussed. The accuracy of the proposed scheme is demonstrated by eight  singular problems contain  various forms of nonlinearity. The numerical results are compared with existing  numerical and exact solutions and it is found that the proposed scheme produce better results. The use of  Haar wavelet, is found to be accurate, fast, flexible, convenient and has small computation costs.

\section{Quasilinearization}
In this section,  the quasilinearization technique \cite{bellman1965quasilinearization} is used to reduce nonlinear DSBVPs  \eqref{sec:eq1} to a sequence of linear problems as
\begin{align}\label{sec2:eq2}
(p(x)y_{n+1}')'=q(x) \left\{f(x,y_{n})+(y_{n+1}(x)-y_{n})f_{y}(x,y_{n})\right\}.
\end{align}
The  sequence of linear problem \eqref{sec2:eq2} may be written  as
\begin{align}\label{sec2:eq3}
p(x) y''_{n+1}+p'(x)y'_{n+1}+r(x)y_{n+1}=g(x),~~~~~n=0,1,2\ldots.
\end{align}
where $r(x)$ and $g(x)$ are given by 
\begin{align*}
r(x)=-q(x)f_y(x,y_n)~~~\hbox{and}~~~g(x)=q(x)(f(x,y_n)-y_n f_y(x,y_n)).
\end{align*}
Boundary conditions \eqref{sec:eq2} and \eqref{sec:eq3} take the form
\begin{align}
\label{sec2:eq4}& y_{n+1}(0)=\alpha_1,~~~~y_{n+1}(1)=\beta_1,\\
\label{sec2:eq5} & y'_{n+1}(0)=0,~~~~~\alpha_2\; y_{n+1}(1)+\beta_2\; y'_{n+1}(1)=\gamma_2.
\end{align}
Integral form of  DSBVPs \eqref{sec2:eq2} with  \eqref{sec2:eq4} and \eqref{sec2:eq5} is given by
\begin{align}\label{sec2:eq6}
y_{n+1}=v(x)+\int\limits_{0}^{1}G(x,t) q(t) \left\{f(t,y_{n})+(y_{n+1}-y_{n})f_{y}(t,y_{n})\right\} dt,
\end{align}
where  $v(x)$ and $G(x,t)$ corresponding to boundary conditions \eqref{sec2:eq4}, are  given by
\begin{align*}
&v(x)=\alpha_1+(\beta_1-\alpha_1) \frac{b(x)}{b(1)},\\
&G(x,t)=\left   \{
  \begin{array}{ll}
    \frac{b(x)}{b(1)}\big(b(1)-b(t)\big), & \hbox{$0< x \leq  t$}, \\
    \frac{b(t)}{b(1)}\big(b(1)-b(x)\big), & \hbox{$t \leq x\leq 1$}
\end{array}
\right.
\end{align*}
and corresponding to boundary conditions \eqref{sec2:eq5}, are  given by
\begin{align*}
&v(x)=\frac{\gamma_2}{\alpha_2},\\
&G(x,t)=\left   \{
  \begin{array}{ll}
  b(1)-b(t)+\frac{\beta_2 \;b'(1)}{\alpha_2} , & \hbox{$0<x\leq t$},\\
    b(1)-b(x)+\frac{\beta_2\; b'(1)}{\alpha_2} , & \hbox{$t\leq x\leq 1$}
\end{array}
\right.
\end{align*}
where $b(x)=\displaystyle \int\limits_{0}^{x}\frac{dt}{p(t)}$,~~ $b(1)=\displaystyle\int\limits_{0}^{1}\frac{dt}{p(t)}$ and $b'(1)=\displaystyle\frac{1}{p(1)}.$

\section{Derivation of Haar wavelets}

\subsection{Haar wavelets}
The basic idea of  the Haar wavelets and its applications can be found in \cite{hsiao2001haar,hsiao2004haar,lepik2005numerical,lepik2007numerical,ur2012numerical}. The Haar wavelet family defined on the interval $[0, 1)$ consists of the following functions:
 \begin{align}\label{sec3:eq3}
h_1(x)= \left   \{
  \begin{array}{ll}
   1, & ~~0\leq x<1,\\
   0,& \hbox{elsewhere} ,
\end{array}
\right.
\end{align}
and for $i=2,3,\ldots$
\begin{align}\label{sec3:eq1}
h_i(x)= \left   \{
  \begin{array}{ll}
   1, & \xi_1\leq x<\xi_2,\\
   -1,& \xi_2\leq x<\xi_3,\\
    0,& \hbox{elsewhere},
\end{array}
\right.
\end{align}
where
\begin{align*}
\xi_1=\frac{k}{m},~\xi_2=\frac{k+0.5}{m},~\xi_3=\frac{k+1}{m};~~m=2^{j},~~j=0,1,\ldots,~~k=0,1,\ldots,m-1.
\end{align*}
Integer $j$ indicates the level of resolution  the wavelet and  $k$ is the translation parameter.  The relation between $i, m$  and $k$  is given by  $i=m+k+1$.
 Note that the Haar functions may also be constructed from the following  relations
\begin{align}
\label{sec3:eq15a} &h_i(x)=2^{\frac{j}{2}} H(2^j x-k),~~~k=0,1,\ldots,2^j-1,~j=0,1,\ldots \vspace{0.15cm}\\
\label{sec3:eq15b} &H(2^j x-k)= \left   \{
  \begin{array}{ll}
   1, &  \displaystyle \frac{k}{2^{j}}\leq x<\frac{k+\frac{1}{2}}{2^{j}} \vspace{0.2cm}\\
   -1,&  \displaystyle \frac{k+\frac{1}{2}}{2^{j}}\leq x<\frac{k+1}{2^{j}}\\
    0,& \hbox{elsewhere}
\end{array}
\right.
\end{align}
Any function $y (x)\in L^2[0,1]$ may be approximated by a finite sum of Haar wavelets as follows
\begin{align}
 y(x)=\sum_{i=1}^{2M} a_i h_i(x),
\end{align}
where $J$ is the maximum value of $j$ and  $M =2^J$. For simplicity, we introduce the following notation
\begin{align}\label{sec3:eq6}
p_{i, 1}(x)=\int_{0}^{x} h_i(t) dt,~~p_{i, 2}(x)=\int_{0}^{x} p_{i, 1}(t) dt,~~C_{i, 1}=\int_{0}^{1} p_{i, 1}(t) dt.
\end{align}
Integrals \eqref{sec3:eq6}  can be evaluated by using \eqref{sec3:eq1} and given by
\begin{align}\label{sec3:eq7}
p_{i, 1}(x)= \left   \{
  \begin{array}{ll}
   x-\xi_1, & \xi_1\leq x<\xi_2,\\
   \xi_3-x,& \xi_2\leq x<\xi_3,\\
    0,& \hbox{elsewhere} ,
\end{array}
\right.
\end{align}
and
\begin{align}\label{sec3:eq8}
p_{i, 2}(x)= \left   \{
  \begin{array}{ll}
  \displaystyle  \frac{1}{2}(x-\xi_1)^2,   & \xi_1\leq x<\xi_2, \vspace{0.2cm}\\
   \displaystyle\frac{1}{4m^2} -\frac{1}{2} (\xi_3-x)^2,  & \xi_2\leq x<\xi_3, \vspace{0.2cm}\\
   \displaystyle\frac{1}{4m^2},  & \xi_3\leq x<1,\vspace{0.2cm}\\
    0,& \hbox{elsewhere} ,
\end{array}
\right.
\end{align}
\noindent Haar wavelet functions satisfy the following properties
\begin{align}
\int_{0}^{1} h_i(x) h_l(x) dx= \left \{
  \begin{array}{ll}
   2^{-j}, & i=l=2^{j}+k,\\
   0,& i\neq l.
\end{array}
\right.
\end{align}
and
\begin{align}\label{sec3:eq5}
\int_{0}^{1} h_{i}(x) dx= \left   \{
  \begin{array}{ll}
   1, & \hbox{if} ~i=1,\\
   0,& \hbox{if}~~i=2,3,\ldots
\end{array}
\right.
\end{align}

%%%%%%%%%%%%%%%%%%%%%%%%%%%%%%%%%%%%%%%%%%%%%%%%%%%%%%%%%%%%%%%%%%%%%%%
\subsection{Haar wavelet quasilinearization technique}
The Haar wavelet   quasilinearization technique  will be discussed for \eqref{sec2:eq3} with BCs.  \eqref{sec2:eq4} and \eqref{sec2:eq5}.
\subsubsection{The Dirichlet BCs. \eqref{sec2:eq4}}
To apply the Haar wavelet  to  problem \eqref{sec2:eq3}, we approximate the second  order derivative term by the Haar wavelet series as
\begin{align}\label{sec3:eq9}
y''_{n+1}(x)=\sum_{i=1}^{2M} a_i h_i(x).
\end{align}
Let us define the collocation points as
\begin{align}\label{sec3:eq10}
x_j=\frac{j-0.5}{2M},~~~~~~j=1,2,\ldots,2M.
\end{align}
Integrating  \eqref{sec3:eq9} twice from $0$ to $x$ and using BCs. \eqref{sec2:eq4}, we get
\begin{align}
\label{sec3:eq11}  y_{n+1}(x)&=\alpha_1+(\beta_1-\alpha_1)x +\sum_{i=1}^{2M} a_i \bigg(p_{i,2}(x)-C_{i,1}\;x\bigg),\\
\label{sec3:eq11a} y'_{n+1}(x)&=\beta_1-\alpha_1+\sum_{i=1}^{2M} a_i \bigg(p_{i,1}(x)-C_{i,1}\bigg).
\end{align}
Substituting \eqref{sec3:eq9},  \eqref{sec3:eq11} and  \eqref{sec3:eq11a} into \eqref{sec2:eq3} and  inserting  collocation points \eqref{sec3:eq10}, a linear system of algebraic equations is obtained as
\begin{align}\label{sec3:eq12}
p(x_j)y''_{n+1}(x_j)+p'(x_j)y'_{n+1}(x_j)+r(x_j)y_{n+1}(x_j)=g(x_j),
\end{align}
where $~n=0,1\ldots,~j=1,2\ldots,2M,$ and
\begin{align}
\label{sec3:eq12a} y''_{n+1}(x_j)&=\sum_{i=1}^{2M} a_i h_i(x_j),\\
\label{sec3:eq12b} y'_{n+1}(x_j)&=\beta_1-\alpha_1+\sum_{i=1}^{2M} a_i \bigg(p_{i,1}(x_j)-C_{i,1}\bigg),\\
\label{sec3:eq12c} y_{n+1}(x_j)&=\alpha_1+(\beta_1-\alpha_1) x_j +\sum_{i=1}^{2M} a_i \bigg(p_{i,2}(x_j)-C_{i,1} x_j \bigg),\\
\label{sec3:eq12d} r(x_j)&=-q(x_j) f_y(x_j,y_n(x_j)),\\
\label{sec3:eq12e} g(x_j)&=q(x_j)\bigg(f(x_j,y_n(x_j))-y_n(x_j) f_y(x_j,y_n(x_j))\bigg).
\end{align}
Equation \eqref{sec3:eq12} gives a sequence of  $2M\times2M$ linear system of equations whose solution for the unknown coefficients $a_1,a_2,\ldots,a_{2M}$ can be calculated using the Gauss-elimination  method. We start with an initial approximation $y_{0}$ to get solutions $y_{1},y_{2}\ldots$.

\subsubsection{The Neumann-Robin BCs. \eqref{sec2:eq5}}
We integrate \eqref{sec3:eq9} twice from $0$ to $x$ and apply BCs. \eqref{sec2:eq5} to get
\begin{align}
\label{sec3:eq13} y_{n+1}(x)&=\frac{\gamma_2}{\alpha_2}-\frac{\beta_2}{\alpha_2} a_1+\sum_{i=1}^{2M} a_i \bigg(p_{i,2}(x)-C_{i,1}\bigg),\\
\label{sec3:eq13a} y'_{n+1}(x)&=\sum_{i=1}^{2M} a_i p_{i,1}(x).
\end{align}
Substituting  \eqref{sec3:eq9},  \eqref{sec3:eq13} and  \eqref{sec3:eq13a} into \eqref{sec2:eq3} and  inserting  collocation points \eqref{sec3:eq10}, a linear system of algebraic equations is obtained as
\begin{align}\label{sec3:eq14}
p(x_j)y''_{n+1}(x_j)+p'(x_j)y'_{n+1}(x_j)+r(x_j)y_{n+1}(x_j)=g(x_j),
\end{align}
where $n=0,1\ldots,~j=1,2\ldots,2M,$ and
\begin{align}
\label{sec3:eq14a} y''_{n+1}(x_j)&=\sum_{i=1}^{2M} a_i h_i(x_j),\\
\label{sec3:eq14b} y'_{n+1}(x_j)&=\sum_{i=1}^{2M} a_i p_{i,1}(x_j),\\
\label{sec3:eq14c} y_{n+1}(x_j)&=\frac{\gamma_2}{\alpha_2}-\frac{\beta_2}{\alpha_2} a_1+\sum_{i=1}^{2M} a_i \bigg(p_{i,2}(x_j)-C_{i,1}\bigg),\\
\label{sec3:eq14d} r(x_j)&=-q(x_j) f_y((x_j),y_n(x_j)),\\
\label{sec3:eq14e} g(x_j)&=q(x_j)\bigg(f(x_j,y_n(x_j))-y_n(x_j) f_y(x_j,y_n(x_j))\bigg).
\end{align}
Equation \eqref{sec3:eq14} gives a sequence of  $2M\times2M$  linear system of equations whose solution for the unknown coefficients $a_1,a_2,\ldots,a_{2M}$ can be calculated
using Gauss-elimination  method. We start with an initial approximation $y_{0}$ to get solutions $y_{1},y_{2}\ldots$.

\section{Convergence}
The present work is based on quasilinearization technique and Haar wavelet method, so  we discuss the convergence of
both the schemes.

\subsection{Convergence of quasilinearization technique}

\begin{theorem}\label{sec2:eq7}
The sequence  $\{y_{n}\}$ of solutions defined in \eqref{sec2:eq6} converges uniformly with quadratic  rate of  convergence.
 \end{theorem}
\begin{proof}
From \eqref{sec2:eq6},  we have
\begin{eqnarray}\label{sec2:eq9}
\Delta y_{n+1}=\int \limits_{0}^{ 1} G(x,t)q(t) \left\{\Delta f_n+\Delta y_{n+1}f_{y}(t,y_{n})-\Delta y_{n}f_{y}(t,y_{n-1})\right\}dt,
\end{eqnarray}
where $\Delta y_n= y_n-y_{n-1}$ and  $\Delta f_n=f_n(x,y_n)-f_{n-1}(x,y_{n-1})$.

\noindent From the mean-value theorem we know  that
\begin{align}\label{sec2:eq10}
\Delta f_n=\Delta y_{n}f_{y}(x,y_{n-1})+\frac{(\Delta y_{n})^{2}}{2} f_{yy}(x,\theta),~~~~y_{n-1}<\theta<y_{n}.
\end{align}
Using   \eqref{sec2:eq10}, the \eqref{sec2:eq9} reduces to
\begin{align}\label{sec2:eq11}
\Delta y_{n+1}= \int\limits_{0}^{1}G(x,t)q(t)\left\{\frac{(\Delta y_{n})^{2}}{2} f_{yy}(t,\theta)+\Delta y_{n+1}f_{y}(t,y_{n})\right\}dt.
\end{align}
Equation \eqref{sec2:eq11} implies
\begin{align}\label{sec2:eq12}
|\Delta y_{n+1}| \leq \int \limits_{0}^{1}|G(x,t)q(t)|\left\{\frac{| \Delta y_{n}|^{2}}{2}|f_{yy}(t,\theta)|+|\Delta y_{n+1}| |f_{y}(t,y_{n})|\right\}dt.
\end{align}
Hence, we have
\begin{align*}
\|\Delta y_{n+1}\|& \leq \max_{x \in [0,1]}\int \limits_{0}^{1}|G(x,t)q(t)|\left\{\frac{| \Delta y_{n}|^{2}}{2}|f_{yy}(t,\theta)|+|\Delta y_{n+1}| |f_{y}(t,y_{n})|\right\}dt \\
&\leq \frac{k_1\; g_1}{2}\| \Delta y_{n}\|^{2} +m_1 g_1 \|\Delta y_{n+1}\|,
\end{align*}
where \begin{align*}
\|y\|&= \max_{x\in [0,1]}|y(x)|,~~~k_1= \max_{y}|f_{yy}(y)|,~~~m_1=\max_{y}\{|f(y)|,|f_y(y)|\}<\infty,\\
g_1&=\max_{x\in [0,1]}\bigg|\int \limits_{0}^{ 1}G(x,t)q(t)dt\bigg|<\infty.
\end{align*}
Thus we have
\begin{align}\label{sec2:eq13}
\|\Delta y_{n+1}\|\leq \frac{k_1\; g_1}{2(1-m_1 g_1)}\|\Delta y_{n}\|^2.
\end{align}
This shows that there is quadratic convergence, if there is convergence at all.
\end{proof}

\subsection{Convergence of Haar wavelet method}

\begin{theorem}\label{sec3:eq15}
Suppose that $y(x)\in\mathcal{L}^2[0,1]$ and $y(x)$ satisfies Lipschitz's condition $|y(x_1)-y(x_2)|\leq L |x_1-x_2|.$  Then Haar wavelet method will be convergent in the sense of $\|y-y_k\|_{2}\rightarrow 0$ as $k\rightarrow \infty$ and its order of convergence is
 \begin{align}\label{sec3:eq16}
\|y-y_k\|_{2}=O\bigg(\frac{1}{ k}\bigg),
\end{align}
where $k=2M=2^{J+1}$, $J$ is the level of resolution  the Haar wavelet.
\end{theorem}
\begin{proof}
Consider
\begin{align*}
y(x)-y_k(x)=\sum_{i=k}^{\infty} a_i h_i(x)=\sum_{i=2^{J+1}}^{\infty} a_i h_i(x),
\end{align*}
where, ~~$y_k(x)= \displaystyle \sum_{i=0}^{k-1} a_i h_i(x)$.
Taking $\mathcal{L}_2$ norm, we obtain
\begin{align}\label{sec3:eq17a}
\nonumber  \|y-y_k\|_{2}^2&=\int\limits_{0}^{1}\bigg(\sum_{i=2^{J+1}}^{\infty} a_i h_i(x), \sum_{l=2^{J+1}}^{\infty} a_l h_l(x) \bigg)dx=\sum_{i=2^{J+1}}^{\infty} \sum_{l=2^{J+1}}^{\infty}a_i  a_l \int\limits_{0}^{1} h_i(x) h_l(x)dx\\
                & =\sum_{i=2^{J+1}}^{\infty} a_i^2,
\end{align}
where $a_i$ is
 \begin{align}\label{sec3:eq17}
a_i=\int\limits_{0}^{1} y(x) h_i(x)dx.
\end{align}
Using \eqref{sec3:eq15a} and \eqref{sec3:eq15b},   the equation \eqref{sec3:eq17} can be written as
\begin{align}\label{sec3:eq18}
a_i=2^{\frac{j}{2}}\left(\int\limits_{\frac{k}{2^{j}}}^{\frac{k+\frac{1}{2}}{2^{j}}} y(x) dx-\int\limits_{\frac{k+\frac{1}{2}}{2^{j}}}^{\frac{k+1}{2^{j}}} y(x) dx\right).
\end{align}
Applying the mean value  theorem and  Lipschitz's condition, the equation \eqref{sec3:eq18} becomes
\begin{align*}
a_i&=2^{\frac{j}{2}} \bigg[ \bigg(\frac{k+\frac{1}{2}}{2^{j}}-\frac{k}{2^{j}}\bigg) y(x_1)-\bigg(\frac{k+1}{2^{j}}-\frac{k+\frac{1}{2}}{2^{j}}\bigg) y(x_2)\bigg]\\
&=2^{-\frac{j}{2}-1}[y(x_1)-y(x_2)]\leq 2^{-\frac{j}{2}-1} L (x_1-x_2)=2^{-\frac{j}{2}-1}\; L\; 2^{-j}=L2^{- \frac{3j}{2}-1}
\end{align*}
Thus, we obtain
\begin{align}\label{sec3:eq18a}
 a_i\leq L 2^{- \frac{3j}{2}-1}.
\end{align}
Using \eqref{sec3:eq18a} into \eqref{sec3:eq17a}, we get
\begin{align*}
\|y-y_k\|_{2}^2&=\sum_{i=2^{J+1}}^{\infty} a_i^2= \sum_{j=J+1}^{\infty} \bigg[\sum_{i=2^{j}}^{2^{j+1}-1} a_i^2\bigg]\leq \sum_{j=J+1}^{\infty} L^2\bigg[\sum_{i=2^{j}}^{2^{j+1}-1} 2^{-3j-2} \bigg]\\
                 &= L^2 \sum_{j=J+1}^{\infty} 2^{-3j-2}[ 2^{j+1}-1-2^{j}+1]=\frac{L^2 }{3} 2^{-2-2 J}=\frac{L^2 }{3}\frac{1}{k^2}.
\end{align*}
Hence, we obtain
\begin{align}\label{sec3:eq19}
\|y-y_k\|_2= O\bigg(\frac{1}{k}\bigg).
\end{align}
 Equation  \eqref{sec3:eq19} ensures the convergence of Haar wavelet approximation at higher level of resolution $J$ is considered.
 \end{proof}
\noindent \textbf{Remark}:  Each iteration of quasilinearization technique gives linear singular  equation in $y_{n+1}$ which is solved to obtain $y_{n+1}$ by Haar wavelet method. According to  \eqref{sec3:eq19}, $y_{n+1}$
converges  to $y$  if  the higher level of resolution $J$ is considered, and at the same time quasilinearization technique works that is for given $y_{0}$, we obtain solution $y_{1}$ of linear problem  \eqref{sec2:eq3} with BCs.  \eqref{sec2:eq4} and \eqref{sec2:eq5} by Haar wavelet method, at next iteration we get $y_{2}$ by Haar wavelet method and so on. Since quasilinearization technique is second order accurate so it gives rapid convergence.

\section{Numerical experiments and discussion}
To check the accuracy and efficiency of  Haar wavelet quasilinearization technique eight singular problems are considered. For the sake of comparison, the cubic spline interpolation is used to obtain the solution at any points in the interval $[0,1]$.  We define absolute error as
$$e_a=|y-y_h|$$
where $y$ is exact  and  $y_h$  Haar solutions. All computational work has been done with the help of MATLAB software.

\begin{problem}\label{sec4:eq1}
\end{problem}
\noindent Consider the problem  \eqref{sec:eq1} with BCs.  \eqref{sec:eq2} where  $p(x)=q(x)=x^{0.5}$ and  $f(x,y)=0.5 e^{y}-e^{2y}$ as in
 \cite{m2003decomposition,singh2013numerical}.  Its exact solution is $$y=\ln \bigg(\frac{2}{x^2+1}\bigg).$$
 Here,  $\alpha_1=\ln2$ and $\beta_1=0$. We have solved this problem by the Haar wavelet quasilinearization technique. By fixing $J=3$ with $n=8$ iterations, we obtain Haar solution $y_h$.  The numerical results of the Haar solution $y_h$ with those  the exact $y$ and the ADM with Green's (ADMG)  \cite{singh2013numerical} along with the maximum absolute error $e_a$ are reported in Table \ref{tab1}. The graphs of the exact $y$  and hwcm  solutions  are depicted Fig. \ref{fig1}.
\begin{figure}[!htb]
\centering
\includegraphics[width=0.45\textwidth]{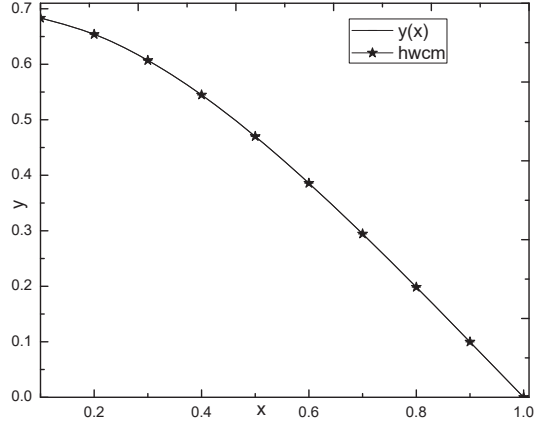}  \vspace{-0.3cm}
\caption{Plots  of the Haar   and  the exact $y$ solutions of problem \ref{sec4:eq1} for $J=3$}\label{fig1}
\end{figure}
\begin{table}[!htb]
\caption{Numerical solutions and  the absolute error $e_a=|y-y_h|$ of problem \ref{sec4:eq1}}\label{tab1}
 \vspace{-0.3cm}
\centering
\addtolength{\tabcolsep}{15pt}
\begin{tabular}{c|cc ccc}
\cline{1-5}
$x$ & $y_h$ & $y$ & ADMG \cite{singh2013numerical} & $e_a$	\\
\cline{1-5}
0.1	&	0.68320	&	0.68330	&	0.68367&	1.06E-04	\\
%0.2	&	0.65393	&	0.65408	&	0.65461&	1.50E-04	\\
0.3	&	0.60697	&	0.60715	&	0.60784&	1.79E-04	\\
%0.4	&	0.54473	&	0.54492	&	0.54579&	1.93E-04	\\
0.5	&	0.47020	&	0.47020	&	0.47125&	1.92E-04	\\
%0.6	&	0.38566	&	0.38584	&	0.38707&	1.74E-04	\\
0.7	&	0.29437	&	0.29452	&	0.29585&	1.44E-04	\\
%0.8	&	0.19845	&	0.19855	&	0.19978&	1.03E-04	\\
0.9	&	0.09982	&	0.099875&	0.10065&	5.42E-05	\\
%1.0	&	0.00001	&	1.07E-05&	1.11E-16&	1.07E-05	\\
\hline
\end{tabular}
\end{table}

%%%%%%%%%%%%%%%%%%%%%%%%%%%%%%%%%%%%%%%%%%
\begin{problem}\label{sec4:eq2}
\end{problem}
\noindent Consider the problem  \eqref{sec:eq1} with BCs.  \eqref{sec:eq2} where  $p(x)=1$,  $q(x)=x^{-\frac{1}{2}}$ and $f(x,y)=y^{3/2}$ as in
\cite{thomas1927calculation,fermi1927metodo} known as Thomas-Fermi equation. Here, $\alpha_1=1$ and $\beta_1=0$. Similarly, we have solved this problem by the Haar wavelet quasilinearization technique. For $J=2,~3$  with $n=7$ iterations, we obtain Haar solution $y_h$. The numerical results of the Haar solution $y_h$ at $J=2,3$ and the ADM  \cite{singh2013solving} in Table \ref{tab2}. We also plot the  graphs of the ADM solution  and haar solution  in Fig. \ref{fig2}.

\begin{figure}[htbp]
\centering
\includegraphics[width=0.45\textwidth]{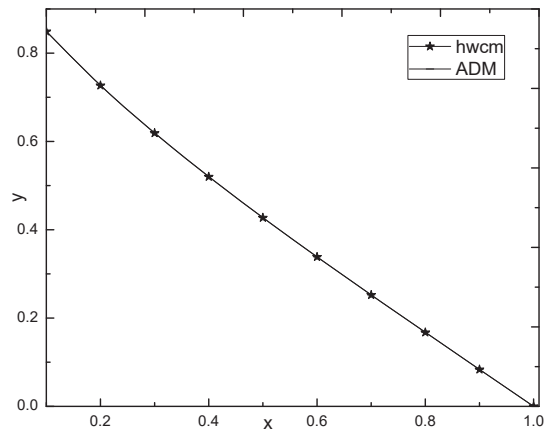}
 \vspace{-0.3cm}
\caption{Plots of the Haar and  the ADM solutions of Problem \ref{sec4:eq2} for $J=2$}
\label{fig2}
\end{figure}

\begin{table}[htbp]
\caption{Numerical solutions and  the absolute error $e_a=|y-y_h|$ of problem \ref{sec4:eq2}}\label{tab2}
 \vspace{-0.3cm}
\centering
 \addtolength{\tabcolsep}{16pt}
\begin{tabular}{c|cc ccc}
\cline{1-4}
$x$& $y_h~(J=2)$& $y_h~(J=3)$ & ADM  \cite{singh2013solving} \\
\cline{1-4}
0.1	&	0.84976	&	0.84909	&	0.84950\\
%0.2	&	0.72627	&	0.72680	&	0.72729\\
0.3	&	0.61829	&	0.61888	&	0.61937	\\
%0.4	&	0.51951	&	0.52004	&	0.52051	\\
0.5	&	0.42672	&	0.42723	&	0.42765	\\
%0.6	&	0.33801	&	0.33842	&	0.33879	\\
0.7	&	0.25187	&	0.25220	&	0.25249	\\
%0.8	&	0.16728	&	0.16751	&	0.16773	\\
0.9	&	0.08351	&	0.08361	&	0.08374	\\
%1.0	&	0.00011	&	7.57E-06&	0.00000\\
\hline
\end{tabular}
\end{table}

%%%%%%%%%%%%%%%%%%%%%%%%%%%%%%%%%%%%%%%%%%%%%%%%%%%%%%%%%%

\begin{problem}\label{sec4:eq3}
\end{problem}
\noindent Consider the problem  \eqref{sec:eq1} with BCs.  \eqref{sec:eq2} where  $p(x)=q(x)=x^{0.5}$ and $f(x,y)=4 x^{2}e^{y}\big( 4x^{4}e^{y}-3.5 \big)$  as in \cite{aziz2001fourth,singh2013numerical,singh2016efficient}.  The exact solution is
$$y(x)=\ln \bigg(\frac{1}{4+x^{4}}\bigg).$$ Here, $\alpha_1=\ln\frac{1}{4}$ and  $\beta_1=\ln\frac{1}{5}$. By fixing $J=3$ with $n=7$ iterations, we obtain Haar solution $y_h$. In  Table \ref{tab3}, the numerical results of Haar solution $y_h$ with those the exact $y$ and ADMG solution \cite{singh2013numerical} and the absolute error $e_a$ are shown. The graphs of the exact $y$  and hwcm  solutions  are plotted in Fig. \ref{fig3}.
\begin{figure}[!htb]
\centering
\includegraphics[width=0.45\textwidth]{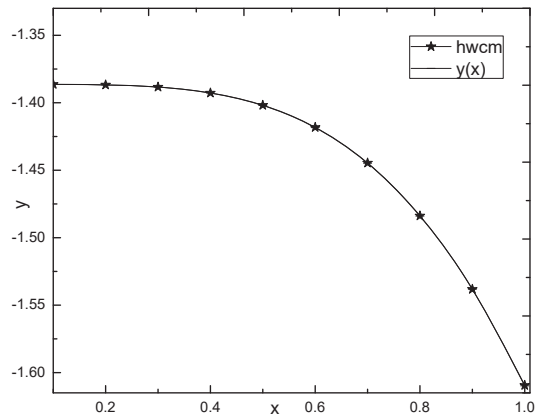} \vspace{-0.3cm}
\caption{Plots of the  hwcm and  the exact $y$  solutions of problem \ref{sec4:eq3}}
\label{fig3}
\end{figure}
\begin{table}[!htb]
\caption{Numerical solutions and the absolute error $e_a$ of problem \ref{sec4:eq3}}\label{tab3}
 \vspace{-0.3cm}
\centering
\addtolength{\tabcolsep}{14pt}
\begin{tabular}{c|cc ccc}
\cline{1-5}
$x$& $y_h$ & $y$ & ADMG \cite{singh2013numerical}& $e_a$	\\
\cline{1-5}
0.1	&	-1.38630	&	-1.38640	&	-1.38632	&	5.92E-05	\\
%0.2	&	-1.38670	&	-1.38670	&	-1.38669	&	3.34E-05	\\
0.3	&	-1.38830	&	-1.38840	&	-1.38832	&	6.90E-05	\\
%0.4	&	-1.39270	&	-1.39280	&	-1.39267	&	1.12E-04	\\
0.5	&	-1.40180	&	-1.40200	&	-1.40180	&	1.57E-04	\\
%0.6	&	-1.41820	&	-1.41840	&	-1.41818	&	1.97E-04	\\
0.7	&	-1.44460	&	-1.44480	&	-1.44459	&	2.22E-04	\\
%0.8	&	-1.48380	&	-1.48400	&	-1.48378	&	2.05E-04	\\
0.9	&	-1.53820	&	-1.53820	&	-1.53818	&	6.86E-05	\\
%1.0	&	-1.60940	&	-1.60980	&	-1.60944	&	3.96E-04	\\
\hline
\end{tabular}
\end{table}

%%%%%%%%%%%%%%%%%%%%%%%%%%%%%%%%%%%%%%%%%%%%%%%%%%%%%%%%%%%%%%
\begin{problem}\label{sec4:eq4}
\end{problem}
\noindent Consider the problem  \eqref{sec:eq1} with BCs.  \eqref{sec:eq3} where $p(x)=q(x)=x^{2}$ and   $f(x,y)=y^{5}$  which describes
the equilibrium of isothermal gas spheres \cite{Chawla1988}.  Its  exact solution is $$y(x)= \bigg(1+ \frac{x^2}{3}\bigg)^{-1/2}.$$ Here,   $\alpha_2=1$, $\beta_2=0$ and $\gamma_2=(3/4)^{1/2}$. By fixing $J=3$  with $n=6$ iterations, we obtain Haar solution $y_h$. The numerical results of Haar solution $y_h$  with those the exact $y$ and the ADMG  \cite{singh2014efficient} and the absolute error $e_a$ are reported in Table \ref{tab4}. The graphs of the exact $y$  and Haar solutions are depicted Fig. \ref{fig4}.

\begin{figure}[!htb]
\centering
\includegraphics[width=0.45\textwidth]{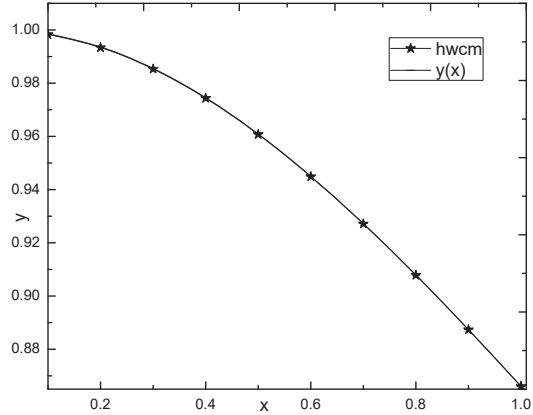}  \vspace{-0.3cm}
\caption{Plots  of the  Haar and  the exact $y$ solutions of Problem  \ref{sec4:eq4}}
\label{fig4}
\end{figure}
\begin{table}[!htb]
\caption{Numerical solutions and  the absolute error $e_a$ of Problem  \ref{sec4:eq4}}\label{tab4}
 \vspace{-0.3cm}
\centering
 \addtolength{\tabcolsep}{8pt}
\begin{tabular}{c|cc ccc}
\cline{1-5}
$x$& $y_h$ &  $y$ & ADMG \cite{singh2014efficient}& $e_a$	\\
\cline{1-5}
0.1	&	0.99834	&	0.99858	&	0.99795	&	2.41E-04	\\
%0.2	&	0.99340	&	0.99362	&	0.99304	&	2.25E-04	\\
0.3	&	0.98533	&	0.98554	&	0.98501	&	2.11E-04	\\
%0.4	&	0.97435	&	0.97455	&	0.97408	&	1.90E-04	\\
0.5	&	0.96077	&	0.96093	&	0.96055	&	1.65E-04	\\
%0.6	&	0.94491	&	0.94505	&	0.94474	&	1.36E-04	\\
0.7	&	0.92715	&	0.92725	&	0.92703 &	1.04E-04	\\
%0.8	&	0.90784	&	0.90791	&	0.90777	&	6.99E-05	\\
0.9	&	0.88736	&	0.88739	&	0.88732	&	3.27E-05	\\
%1.0	&	0.86603	&	0.86604	&	0.86602	&	1.29E-05	\\
\hline
\end{tabular}
\end{table}

%%%%%%%%%%%%%%%%%%%%%%%%%%%%%%%%%%%%%%%%%%%%%%%%%%%%%%%%%%%%%%%%%%%%%%%%%%%%%%

\begin{problem}\label{sec4:eq5}
\end{problem}
\noindent Consider the problem  \eqref{sec:eq1} with BCs.  \eqref{sec:eq3} where $p(x)=q(x)=x$ and  $f(x,y)=-e^{y}$  which arises an electro-hydrodynamics problem \cite{Keller1955}.  Its exact solution is $$y(x)=2\ln\left(\frac{A+1}{A x^2+1}\right),~A=3-2\sqrt{2}.$$ Here, $\alpha_2=1$, $\beta_2=0$ and $\gamma_2=0$. By fixing $J=3$ with $n=6$ iterations, we obtain Haar solution $y_h$. Table \ref{tab5} shows the comparison of Haar solution $y_h$  with those the exact $y$ and the ADMG solutions \cite{singh2014efficient} and absolute error $e_a$. The graphs of the exact $y$  and hwcm solutions  are depicted Fig. \ref{fig5}.

\begin{figure}[!htb]
\centering
\includegraphics[width=0.45\textwidth]{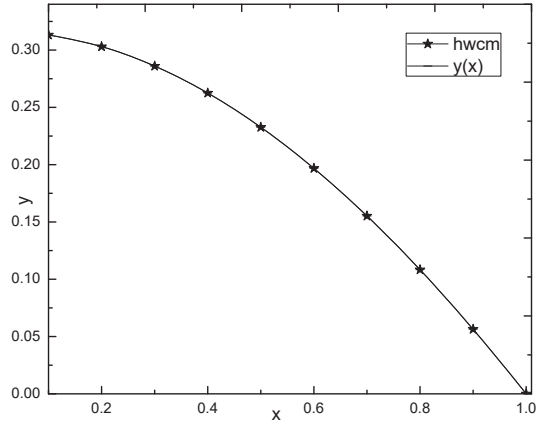}  \vspace{-0.3cm}
\caption{Plots of the hwcm and   exact $y(x)$ solutions of Problem  \ref{sec4:eq5}} \label{fig5}
\end{figure}

\begin{table}[!htb]
\caption{Numerical solutions  and the absolute error $e_a$ of Problem  \ref{sec4:eq5}}\label{tab5}
 \vspace{-0.3cm}
\centering
 \addtolength{\tabcolsep}{12pt}
\begin{tabular}{c|cc ccc}
\cline{1-5}
$x$& $y_h$ &  $y$  & ADMG \cite{singh2014efficient} & $e_a$	\\
\cline{1-5}
0.1	&	0.31327	&	0.31327	&	0.31326	&	8.34E-06	\\
%0.2	&	0.30302	&	0.30302	&	0.30301	&	8.29E-06	\\
0.3	&	0.28605	&	0.28606	&	0.28604	&	8.08E-06	\\
%0.4	&	0.26253	&	0.26254	&	0.26253	&	7.76E-06	\\
0.5	&	0.23270	&	0.23270	&	0.23269	&	7.22E-06	\\
%0.6	&	0.19683	&	0.19683	&	0.19682	&	6.47E-06	\\
0.7	&	0.15525	&	0.15525	&	0.15525	&	5.41E-06	\\
%0.8	&	0.10832	&	0.10833	&	0.10832	&	4.00E-06	\\
0.9	&	0.05643	&	0.05644	&	0.05644	&	2.21E-06	\\
%1.0	&	0.00000	&	2.12E-07&	3.08E-17&	2.12E-07	\\
\hline
\end{tabular}
\end{table}

%%%%%%%%%%%%%%%%%%%%%%%%%%%%%%%%%%%%%%%%%%%%%%%%%%%%%%%%%%%%%%%%%%%%%%%%%%%
\begin{problem}\label{sec4:eq6}
\end{problem}
\noindent Consider the problem  \eqref{sec:eq1} with BCs.  \eqref{sec:eq3} where $p(x)=q(x)=x^{2}$ and  $f(x,y)=-e^{-y}$  which arises in
the distribution of heat sources in the human head \cite{duggan1986pointwise}. Here,  $\alpha_2=2$, $\beta_2=1$ and $\gamma_2=0$.  By fixing $J=3$ with $n=7$ iterations, we obtain Haar solution $y_h$. The comparison of Haar solution $y_h$  with those obtained by  the ADMG  \cite{singh2014efficient}, finite difference method (FDM) 	 \cite{pandey1997finite} and the  tangent chord method (TCM)  \cite{duggan1986pointwise} is presented in Table \ref{tab6}. The graphs of the ADMG  and hwcm  are depicted Fig. \ref{fig6}.
\begin{figure}[!htb]
\centering
\includegraphics[width=0.45\textwidth]{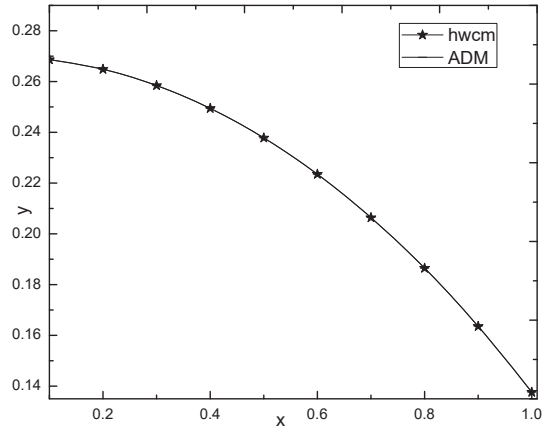} \vspace{-0.3cm}
\caption{Plots of the Haar and   the ADMG  solutions of Problem  \ref{sec4:eq6}} \label{fig6}
\end{figure}
\begin{table}[!htb]
\caption{Numerical solutions  of Problem  \ref{sec4:eq6}}\label{tab6}
 \vspace{-0.3cm}
\centering
 \addtolength{\tabcolsep}{12pt}
\begin{tabular}{c|cc ccc}
\cline{1-5}
$x$& $y_h$ & ADMG \cite{singh2014efficient}	 & TCM \cite{duggan1986pointwise} & FDM	 \cite{pandey1997finite}\\
\cline{1-5}
0.1	&	0.26866	&	0.26862	&	0.26907	&	0.26875	\\
%0.2	&	0.26484	&	0.26480	&	0.26525	&	0.26493	\\
0.3	&	0.25845	&	0.25841	&	0.25886	&	0.25853	\\
%0.4	&	0.24945	&	0.24943	&	0.24986	&	0.24954	\\
0.5	&	0.23782	&	0.23781	&	0.23822	&	0.23791	\\
%0.6	&	0.22349	&	0.22349	&	0.22388	&	0.22358	\\
0.7	&	0.20640	&	0.20641	&	0.20677	&	0.20649	\\
%0.8	&	0.18646	&	0.18648	&	0.18679	&	0.18655	\\
0.9	&	0.16356	&	0.16359	&	0.16387	&	0.16365	\\
%1.0	&	0.13761	&	0.13764	&	0.13787	&	0.13769	\\
\hline
\end{tabular}
\end{table}

%%%%%%%%%%%%%%%%%%%%%%%%%%%%%%%%%%%%%%%%%%%%%%%%%%%%%%%%%%%%%%%%%%%%%%%%%%%%%%%%%%%%%%%%%%%%%%%%%
\begin{problem}\label{sec4:eq7}
\end{problem}
\noindent Consider the problem  \eqref{sec:eq1} with BCs.  \eqref{sec:eq3} where $p(x)=q(x)=x^{2}$ and  $f(x,y)=\frac{0.76129 y}{y+0.03119}$ which models a oxygen diffusion in a spherical cell with  oxygen uptake kinetics \cite{lin1976oxygen}. Here, $\alpha_2=5$, $\beta_2=1$ and  $\gamma_2=5$. By fixing $J=3$  with $n=7$ iterations, we obtain haar solution $y_h$. Table \ref{tab7} shows the comparison  of Haar solutions  $y_h$ with those obtained by
 the ADMG  \cite{singh2014efficient},  the variational iteration method (VIM)   \cite{wazwaz2011variational}  and the cubic spline method (CSM)  \cite{kanth2006cubic}. The graphs of the ADM  and hwcm  solutions  are depicted Fig. \ref{fig7}.
\begin{figure}[!htb]
\centering
 \vspace{-0.3cm}
\includegraphics[width=0.45\textwidth]{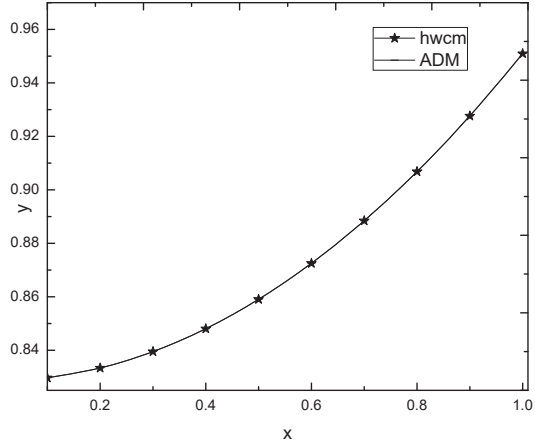}
\caption{Plots of the Haar and  the ADM solutions of Problem  \ref{sec4:eq7}} \label{fig7}
\end{figure}

\begin{table}[!htb]
\caption{Numerical results of solutions of Problem \ref{sec4:eq7}}\label{tab7}
 \vspace{-0.3cm}
\centering
 \addtolength{\tabcolsep}{14pt}
\begin{tabular}{c|cc ccc}
\cline{1-5}
$x$& $y_h$ & ADMG \cite{singh2014efficient}	 & VIM  \cite{wazwaz2011variational}	 & CSM \cite{kanth2006cubic}	\\
\cline{1-5}
0.1	&	0.82971	&	0.82970	&	0.82970	&	0.82970	\\
%0.2	&	0.83338	&	0.83337	&	0.83337 &	0.83337	\\
0.3	&	0.83949	&	0.83949	&	0.83948	&	0.83948	\\
%0.4	&	0.84805	&	0.84805	&	0.84805	&	0.84805	\\
0.5	&	0.85907	&	0.85906	&	0.85906	&	0.85906	\\
%0.6	&	0.87253	&	0.87252	&	0.87252	&	0.87252	\\
0.7	&	0.88845	&	0.88844	&	0.88844	&	0.88844	\\
%0.8	&	0.90682	&	0.90681	&	0.90681	&	0.90681	\\
0.9	&	0.92765	&	0.92765	&	0.92765 &	0.92765\\
%1.0	&	0.95095	&	0.95094	&	0.95094	&	0.95094	\\
\hline
\end{tabular}
\end{table}

%%%%%%%%%%%%%%%%%%%%%%%%%%%%%%%%%%%%%%%%%%%%%%%%%%%%%%%%%%%%%%%%%%
\begin{problem}\label{sec4:eq8}
\end{problem}
\noindent Consider the problem  \eqref{sec:eq1} with BCs.  \eqref{sec:eq3} where $p(x)=q(x)=x^{3}$ and $f(x,y)=\frac{1}{2}-\frac{1}{8y^2}$  which arises in the radial stress on a rotationally symmetric shallow membrane cap \cite{dickey1989rotationally}. Here, $\alpha_2=1$, $\beta_2=0$ and $\gamma_2=1$. By fixing $J=3$  with $n=7$ iterations, we obtain haar solution $y_h$. The comparison of Haar solution $y_h$ with those obtained by the ADMG \cite{singh2014efficient} and  VIM  \cite{ravi2010he} are reported in Table \ref{tab1}. We plot the graphs of the ADM  and hwcm
 solutions in Fig. \ref{fig8}.
 \begin{figure}[!htb]
\centering
\includegraphics[width=0.45\textwidth]{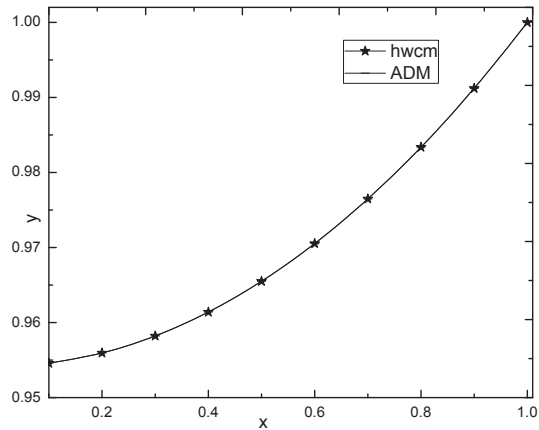} \vspace{-0.3cm}
\caption{Plots of the hwcm  and   the ADM  solutions of Problem\ref{sec4:eq8}} \label{fig8}
\end{figure}
\begin{table}[!htb]
\caption{Numerical results of solutions of Problem  \ref{sec4:eq8}}\label{tab8}
 \vspace{-0.3cm}
\centering
 \addtolength{\tabcolsep}{20pt}
\begin{tabular}{c|cc ccc}
\cline{1-4}
$x$& $y_h$ & ADMG \cite{singh2014efficient}	 & VIM  \cite{ravi2010he}	 \\
\cline{1-4}
0.1	&	0.95459	&	0.95458	&	0.95263	\\
%0.2	&	0.95595	&	0.95594	&	0.95408	\\
0.3	&	0.95822	&	0.95822	&	0.95649	\\
%0.4	&	0.96140	&	0.96140	&	0.95986	\\
0.5	&	0.96550	&	0.96550	&	0.96420	\\
%0.6	&	0.97053	&	0.97052	&	0.96948	\\
0.7	&	0.97648	&	0.97647	&	0.97571	\\
%0.8	&	0.98337	&	0.98336	&	0.98288	\\
0.9	&	0.99121	&	0.99120	&	0.99098	\\
%1.0	&	1.00000	&	1.00000	 &	1.00000	\\
\hline
\end{tabular}
\end{table}

\section{Conclusion}
In this paper, the Haar wavelet quasilinearization technique has been proposed  for
nonlinear doubly singular boundary value  problems arising in various physical models. It has been shown that Haar wavelet method with quasilinearization technique gives excellent results when applied to different physical models such as oxygen diffusion in a spherical cell  \cite{lin1976oxygen}, the heat sources in the human head \cite{duggan1986pointwise},  and shallow membrane cap \cite{dickey1989rotationally}. The numerical results obtained by
present method are better than the results obtained by other methods such as the Adomian decomposition method
 \cite{singh2013numerical,singh2014efficient}, the variational iteration method  \cite{wazwaz2011variational,ravi2010he},
 the finite difference method \cite{pandey1997finite}, the cubic spline method\cite{kanth2006cubic} and the tangent chord method
 \cite{duggan1986pointwise} and are in good agreement with exact solutions, as shown in tables \ref{tab1}-\ref{tab8} and  figures
\ref{fig1}-\ref{fig8} for the considered problems through \ref{sec4:eq1}-\ref{sec4:eq8}. The proposed method provides a reliable technique which requires less work compared to other methods such as the  finite difference and cubic spline methods. The convergence analysis of present methods have been discussed.

%%\bibliographystyle{elsarticle-num}
%%\bibliography{finite-total}
%\bibliographystyle{elsarticle-num}
%\bibliography{randhir_hwcm-1}

\end{document}